\documentclass[sn-mathphys,Numbered]{sn-jnl}

\usepackage{graphicx}%
\usepackage{multirow}%
\usepackage{amsmath,amssymb,amsfonts}%
\usepackage{amsthm}%
\usepackage{mathrsfs}%
\usepackage[title]{appendix}%
\usepackage{xcolor}%
\usepackage{textcomp}%
\usepackage{manyfoot}%
\usepackage{booktabs}%
\usepackage{algorithm}%
\usepackage{algorithmicx}%
\usepackage{algpseudocode}%
\usepackage{listings}%

\newcommand{\Tr}{{\rm Tr}}
\newcommand{\gf}{ {{\mathbb F}} }

\newtheorem{theorem}{Theorem}

\newtheorem{lemma}{Lemma}
\newtheorem{example}{Example}%

\begin{document}

\title[The compositional inverses of permutation polynomials from trace functions over finite fields]{The compositional inverses of permutation polynomials from trace functions over finite fields}

\author*[1]{\fnm{} \sur{Danyao Wu}}\email{wudanyao111@163.com}

\author[2]{\fnm{} \sur{Pingzhi Yuan}}\email{yuanpz@scnu.edu.cn}


\affil*[1]{\orgdiv{School of Computer Science and Technology}, \orgname{Dongguan University of Technology}, \orgaddress{
		\city{Dongguan}, \postcode{523808}, 
		\country{China}}}

\affil[2]{\orgdiv{School of Mathematics}, \orgname{South China Normal University}, \orgaddress{
		\city{Guangzhou}, \postcode{510631},
		\country{China}}}

\abstract{In this paper, we present the compositional inverses of several classes permutation polynomials of the form $\sum_{i=1}^kb_i\left(\Tr_m^{mn}(x)^{t_i}+\delta\right)^{s_i}+f_1(x)$, where $1\leq i \leq k,$ $s_i$ are positive integers, $b_i \in \gf_{p^m},$  $p$ is a  prime and  $f_1(x)$ is a polynomial over $\gf_{p^{mn}}$  satisfying the following conditions:
	(i) $\Tr_m^{mn}(x) \circ f_1(x)=\varphi(x) \circ \Tr_m^{mn}(x),$ where  $\varphi(x)$ is a polynomial over $\gf_{p^m};$
	(ii) For any $a \in \gf_{p^m},$ $f_1(x)$ is injective on $\Tr_m^{mn}(a)^{-1}.$}




\maketitle

\section{Introduction}\label{sec1}

Let  $\gf_q$ be the finite field with $q$ elements,  where $q$ is a prime power, and
let $\gf_q[x]$
be the ring of polynomials in a single indeterminate $x$ over $\gf_q$. A polynomial
$f \in\gf_q[x]$ is called a {\em permutation polynomial} of $\gf_q$ if its
associated polynomial mapping $f: c\mapsto f(c)$ from $\gf_q$ to itself is a bijective. The unique polynomial denoted by $f^{-1}(x)$ over $\gf_q$
such that $f(f^{-1}(x))\equiv f^{-1}(f(x)) \equiv x \pmod{x^q-x}$ is called the compositional inverse of $f(x).$ Furthermore,  $f(x)$ is called  an involution when $f^{-1}(x)=f(x).$
The study of permutation polynomials and their compositional inverses over finite
fields in terms of their coefficients is a classical and difficult subject which
attracts people's interest partially due to their wide applications in coding theory
\cite{ding2013cyclic,ding2014binary,laigle2007permutation},
cryptography \cite{rivest1978method,schwenk1998public}, combinatorial design theory \cite{ding2006family}, and other areas of mathematics and engineering \cite{lidl1997finite,lidl1994introduction}. 
In general, it is difficult to discover new classes of permutation polynomials and computing the coefficients of
the compositional inverse of a permutation polynomial seems to be even more difficult, except for several classical classes such as
monomials, linearized polynomials, Dickson polynomials.
Compositional inverses of several classes of permutation polynomials
in explicit or implicit forms have been investigated in recent years,
which have nice structure.
We refer the readers to \cite{coulter2002compositional,li2019compositional,niu2021finding,tuxanidy2014inverses,tuxanidy2017compositional,wang2007cyclotomic,wang2017note,baofengwu2013compositional,wu2014compositional,wu2013compositional,yuan2022compositional,zheng2019inverses,zheng2019constructions} for more details. 

Zeng et al. \cite{zeng2015permutation} investigated the permutation 
properties exhibited by polynomials of the form
$$(\Tr_m^{nm}(x)+\delta)^s+L(x)$$
over  $\gf_{2^{nm}},$ where   $s=k(2^m\pm1)+1,$  $k, n, m \in \mathbb{Z}^+,$ $\delta \in \gf_{2^{nm}},$ $\Tr_m^{nm}(x)$ is the trace function from $\gf_{2^{nm}}$ to $\gf_{2^m}$
and $L(x)=\Tr_m^{nm}(x)+x$ or $x$. The primary method involved determining the number of solutions of special equations within finite fields. Subsequently, Wu and Yuan  \cite{wu2022further} enhanced  all the findings from \cite{zeng2015permutation} by utilizing  the Akbary--Ghioca--Wang (AGW) criterion to study  a class of permutation polynomials of the form
\begin{equation}\label{1}
	b(\Tr_m^{nm}(x)+\delta)^{1+\frac{i(p^{nm}-1)}{d}}+c(\Tr_m^{nm}(x)+\delta)^{1+\frac{j(p^{nm}-1)}{d}}+h(x)
\end{equation}
over $\gf_{p^{nm}},$ where $p$ is a  prime,  $m, n, i, j, d \in\mathbb{Z}^+$ with $p^m\equiv\pm1\pmod{d},$ $b,c \in \gf_{p^m},$ $\delta \in \gf_{p^{nm}}$ and  $h(x)$ is a polynomial over $\gf_{p^{mn}}$  satisfying the following conditions:
(i) $\Tr_m^{mn}(x) \circ h(x)=\tau(x) \circ \Tr_m^{mn}(x),$ where  $\tau(x)$ is a polynomial over $\gf_{p^m};$
(ii) For any $s \in \gf_{p^m},$ $h(x)$ is injective on $\Tr_m^{mn}(s)^{-1}.$ 
Furthermore,  two classes of permutation polynomials in the form of 
$$x+(\Tr_m^{2m}(x)^k+\delta)^s \,\, \text{and} \,\, x+(\Tr_m^{2m}(x)^k+\delta)^{s_1}+ (\Tr_m^{2m}(x)^k+\delta)^{s_2}$$ over $\gf_{2^{2m}}$ were presented in \cite{li2020some} based on the AGW criterion, where $m, k, s_1, s_2 \in \mathbb{Z}^+.$ 
Inspired by the work of \cite{wu2022further,li2020some},  we contribute to the field by presenting the compositional inverses of several classes permutation polynomials of the form $\sum_{i=1}^kb_i\left(\Tr_m^{mn}(x)^{t_i}+\delta\right)^{s_i}+f_1(x)$, where for $1\leq i \leq k,$ $s_i$ are positive integers, $b_i \in \gf_{p^m},$  $p$ is prime and  $f_1(x)$ is a polynomial over $\gf_{p^{mn}}$  satisfying the following conditions:
(i) $\Tr_m^{mn}(x) \circ f_1(x)=\varphi(x) \circ \Tr_m^{mn}(x),$ where  $\varphi(x)$ is a polynomial over $\gf_{p^m};$
(ii) For any $a \in \gf_{p^m},$ $f_1(x)$ is injective on $\Tr_m^{mn}(a)^{-1}.$

The remainder of this paper is organized as follows. Section 2,  introduces basic concepts and related results. Next, Section 3, presents the investigation of the compositional inverses of permutation polynomials in the form $\sum_{i=1}^kb_i\left(\Tr_m^{mn}(x)+\delta\right)^{s_i}+f_1(x)$ over $\gf_{p^{nm}}.$  Finally,  Section 4,  provides  the investigation of the compositional inverses of permutation polynomials in the form $\sum_{i=1}^kb_i\left(\Tr_m^{2m}(x)^{t_i}+\delta\right)^{s_i}+x$ over $\gf_{2^{2m}}.$

\section{Preliminaries}
In this section, we present some auxiliary results that will be needed in the
sequel.

Let $m$ and $n$ be two positive integers with $m\mid n.$ The $trace function $ $\Tr_m^n(\cdot)$ from $\gf_{p^n}$ to $\gf_{p^m}$ is defined by
$$\Tr_m^n(x)=x+x^{p^m}+x^{p^{2m}}+\cdots+x^{p^{(\frac{n}{m}-1)m}}, \, \,  x \in \gf_{p^n}. $$

The following lemma was developed by Akbary, Ghioca and Wang \cite{akbary2011constructing}. It is called the AGW criterion \cite{Hou2015survey}.

\begin{lemma} \label{AGW} \cite[Lemma 1.1]{akbary2011constructing} Let $A, S$ and $\overline{S}$ be finite sets
	with $\sharp S =\sharp \overline{S}$,
	and let $f(x) : A\longrightarrow A$, $g(x): S\longrightarrow \overline{S}$, $\lambda(x): A\longrightarrow S,$ and $\overline{\lambda}(x):A\longrightarrow \overline{S}$ be maps
	such that $\overline{\lambda}(x)\circ f(x)=g(x)\circ \lambda(x).$
	If both $\lambda(x)$ and $\overline{\lambda}(x)$ are surjective,
	then the following statements are equivalent:\\
	(i) $f(x)$ is bijective (a permutation of $A$); and\\
	(ii) $g(x)$ is bijective from $S$ to $\overline{S}$ and
	$f(x)$ is injective on $\lambda^{-1}(s)$ for each $s \in S.$
\end{lemma}
Yuan \cite{yuan2022compositional} investigated the compositional inverses of a class of permutation polynomials in the additive cases.  
\begin{lemma} \label{agwf-} \cite[Theorem 3.4]{yuan2022compositional} 
	Let $q$ be a prime power, and let $S, \bar{S}$ be subsets $\gf_{q}$ with $\sharp S =\sharp \bar{S}.$
	Let $f(x): \gf_{q}\rightarrow \gf_{q}$, $g(x): S\rightarrow \bar{S}, $  $\lambda(x): \gf_{q}\rightarrow S,$ and $\bar{\lambda}(x): \gf_{q}\rightarrow \bar{S}$ be maps such that both $\lambda(x)$ and $\bar{\lambda}(x)$ are surjective maps and $\bar{\lambda}(x)\circ f(x)= g(x)\circ \lambda(x)$.  Let $f_1(x)$ be a permutation polynomial and $f(x)=f_1(x)+h(\lambda(x))$ is a permutation polynomial over $\gf_{q},$ and let $f_1^{-1}(x)$, $f^{-1}(x)$ and $g^{-1}(x)$ be the compositional inverses of $f_1(x)$, $f(x)$ and $g(x)$, respectively. Then 
	$$f^{-1}(x)=f_1^{-1}(x-h(g^{-1}(\bar{\lambda}(x)))).$$
	
\end{lemma}%
If we compute the compositional inverse of $f(x)=f_1(x)+h(\lambda(x))$ by Lemma \ref{agwf-}, ideally, the polynomial $f_1(x)$ should  be a permutation polynomial. However, observations indicate that certain permutation polynomials in this form fail to satisfying the condition of  $f_1(x)$ being a  permutation polynomial. Therefore, it is necessary to refine the aforementioned lemma. The primary motivation  for this enhancement stems from 
\cite{yuan2022compositional}.

We list the lemma about the dual diagram of AGW criterion  \cite{yuan2022compositional}.
\begin{lemma}\label{agwdual}\cite[Theorem 2.6]{yuan2022compositional} 
	Let the notations be defined as 
	in Lemma \ref{AGW}. If $f(x):  A\longrightarrow A$ is a bijection, $f^{-1}(x)$ and $g^{-1}(x)$ are the compositional inverses of $f(x)$ and $g(x),$ respectively, then $$\lambda(x) \circ f^{-1}(x)= g^{-1}(x)\circ \bar{\lambda}(x).$$ 
\end{lemma}
Now, we refine the Lemma \ref{agwf-} based on the the dual diagram of AGW criterion. 
\begin{lemma}\label{agwf-2}
	Let $q$ be a prime power, and let $S, \bar{S}$ be subsets $\gf_{q}$ with $\sharp S =\sharp \bar{S}.$
	Let $f(x): \gf_{q}\rightarrow \gf_{q}$, $g(x): S\rightarrow \bar{S}, $  $\lambda(x): \gf_{q}\rightarrow S,$ and $\bar{\lambda}(x): \gf_{q}\rightarrow \bar{S}$ be maps such that both $\lambda(x)$ and $\bar{\lambda}(x)$ are surjective maps and $\bar{\lambda}(x)\circ f(x)= g(x)\circ \lambda(x)$.   Let $f_1(x)$ is a polynmial over $\gf_q$ and  $f(x)=f_1(x)+h(\lambda(x))$ is a permutation polynomial over $\gf_{q}.$ If there exist two  polynomials $\phi(x)$ and  $\bar{\phi}(x)$ over $\gf_q$ such that $\phi(x)\circ f_1(x)+\bar{\phi}(x)\circ \lambda(x)$ permutes $\gf_q$,    then the compositional inverse of $f(x)$ over $\gf_q$ is 
	$$f^{-1}(x)=(\phi(f_1(x)) +\bar{\phi}(\lambda(x)))^{-1}\circ \left(\phi(x-h(g^{-1}(\bar{\lambda}(x))) )+\bar{\phi}(g^{-1}(\bar{\lambda}(x))) \right),$$
	where $g^{-1}(x)$  and $(\phi(f_1(x)) +\bar{\phi}(\lambda(x)))^{-1}$ are the compositional inverses of $g(x)$ and $\phi(f_1(x)) +\bar{\phi}(\lambda(x))$, respectively. 
\end{lemma}
\begin{proof}
Given that  $f(x)$ is a permutation polynomial over $\gf_q$ and $f^{-1}(x)$  is the compositional inverse of $f(x)$, we conclude  $f(x)\circ f^{-1}(x)=x$ and  $\lambda(x) \circ f^{-1}(x)= g^{-1}(x)\circ \bar{\lambda}(x)$ by Lemma \ref{agwdual}. 
Then we obtain the system of  the equations 
	\begin{equation*}
		\begin{cases}
			f(x)\circ f^{-1}(x)&=\left(f_1(x)+h(\lambda(x))\right)\circ f^{-1}(x);\\
			\lambda(x) \circ f^{-1}(x)&= g^{-1}(x)\circ \bar{\lambda}(x),
		\end{cases}
	\end{equation*}
	or, equivalently, 
	\begin{equation*}
		\begin{cases}
			f_1(x) \circ f^{-1}(x)&=\,\,	x-h(g^{-1}(\bar{\lambda}(x));\\
			\lambda(x) \circ f^{-1}(x)&=\,\, g^{-1}(x)\circ \bar{\lambda}(x),
		\end{cases}
	\end{equation*}
Consequently, we have 
	\begin{equation}\label{agwf-2eq1}
		\begin{cases}
			\phi(x)\circ f_1(x) \circ f^{-1}(x)&=\,\,	\phi(x)\circ\left(x-h(g^{-1}(\bar{\lambda}(x))\right);\\
			\bar{\phi}(x)\circ	\lambda(x) \circ f^{-1}(x)&=\,\,\bar{\phi}(x)\circ g^{-1}(x)\circ \bar{\lambda}(x).
		\end{cases}
	\end{equation}
	By adding the equations of the system \eqref{agwf-2eq1}, we obtain 
	$$\left(	\phi(f_1(x)) +\bar{\phi}(\lambda(x))\right)\circ  f^{-1}(x)=\phi\left(x-h(g^{-1}(\bar{\lambda}(x))\right)+\bar{\phi}( g^{-1}(\bar{\lambda}(x))).$$
	This implies the desired result, because $\phi(x)\circ f_1(x)+\bar{\phi}(x)\circ \lambda(x)$ permutes $\gf_q.$
\end{proof}

The following lemma is crucial in this paper, which will be frequently used in Sections 3 and 4.
\begin{lemma}\label{th1}
	For a  prime $p$ and $1\leq i \leq k,$    let   $n, m, s_i, t_i$ be  positive integers  and $b_i\in \gf_{p^m}.$ Assume that $f_1(x)$ is a polynomial over $\gf_{p^{mn}}$ and $\varphi(x)$ is a polynomial over $\gf_{p^m}$ satisfying the following conditions:\\ 
	(i) $\Tr_m^{mn}(x) \circ f_1(x)=\varphi(x) \circ \Tr_m^{mn}(x);$\\ 
	(ii) For any $a \in \gf_{p^m},$ $f_1(x)$ is injective on $\Tr_m^{mn}(a)^{-1}.$\\
	Then the polynomial 
	$$f(x)=\sum_{i=1}^kb_i\left(\Tr_m^{mn}(x)^{t_i}+\delta\right)^{s_i}+f_1(x)$$
	permutes $\gf_{p^{mn}}$ if and only if 
	$$g(x)=\sum_{j=0}^{n-1}\sum_{i=1}^kb_i(x^{t_i}+\delta)^{s_ip^{mj}}+\varphi(x)$$ permutes $\gf_{p^m}.$
	Moreover, if $f(x)$ is a permutation polynomial over $\gf_{p^{mn}}$ and  there exist two  polynomials $\phi(x)$ and  $\bar{\phi}(x)$ over $\gf_{p^{nm}}$ such that $\phi(x)\circ f_1(x)+\bar{\phi}(x)\circ \Tr_m^{nm}(x)$ permutes $\gf_{p^{nm}},$ then the compositional inverse of $f(x)$ over $\gf_{p^{mn}}$ is
	$$f^{-1}(x)=(\phi(f_1) +\bar{\phi}(\lambda))^{-1}\circ \left(\phi (x-h(g^{-1}(\Tr_m^{nm}(x))) )+\bar{\phi}(g^{-1}(\Tr_m^{nm}(x))) \right),$$  
	where $h(x)=\sum_{i=1}^kb_i(x^{t_i}+\delta)$, $g^{-1}(x)$  and $(\phi(f_1(x)) +\bar{\phi}(\lambda(x)))^{-1}$ are the compositional inverses of $g(x)$ and $\phi(f_1(x)) +\bar{\phi}(\lambda(x))$, respectively

\end{lemma}
\begin{proof}
	Since $\Tr_m^{mn}(x)$ is  the linearized polynomial over $\gf_{p^{mn}}$ and for $1\leq i \leq k,$ $b_i\in \gf_{p^m}, $ and since moreover $\Tr_m^{mn}(x) \circ f_1(x)=\varphi(x) \circ \Tr_m^{mn}(x),$ we have 
	\begin{align*}
		\Tr_m^{mn}(x)(x)\circ f(x)=&\,\Tr_m^{mn}(x) \circ \sum_{i=1}^kb_i\left(\Tr_m^{mn}(x)^{t_i}+\delta\right)^{s_i}+\Tr_m^{mn}(x) \circ f_1(x)\\
		=&\,\left(\sum_{j=0}^{n-1}\sum_{i=1}^kb_i(x^{t_i}+\delta)^{s_ip^{mj}}+\varphi(x)\right) \circ \Tr_m^{mn}(x)\\
		=&\,g(x)\circ \Tr_m^{mn}(x),
	\end{align*}
	where $g(x)=\sum_{j=0}^{n-1}\sum_{i=1}^kb_i(x^{t_i}+\delta)^{s_ip^{mj}}+\varphi(x).$
	Note that the trace function $ \Tr_m^{mn}(x)$ is a function from $\gf_{p^{mn}}$ onto $\gf_{p^m}.$ 
	For any $a \in \gf_{p^m}$, $f(x)$ is injective on $ \Tr_m^{mn}(a)^{-1}$ because $f_1(x)$ is injective on $\Tr_m^{mn}(a)^{-1}.$
	It follows from Lemma \ref{AGW} that $f(x)$ permutes $\gf_{p^{mn}}$ if and only if $g(x)$ permutes $\gf_{p^{m}}.$

	Furthermore, if $f(x)$ is permutation polynomials over $\gf_{p^{mn}}$ and  there exist two  polynomials $\phi(x)$ and  $\bar{\phi}(x)$ over $\gf_{p^{mn}}$ such that $\phi(x)\circ f_1(x)+\bar{\phi}(x)\circ \Tr_m^{nm}(x)$ permutes $\gf_{p^{nm}}$,  let $g^{-1}(x)$  and $(\phi(f_1(x)) +\bar{\phi}(\lambda(x)))^{-1}$ be the compositional inverses of $g(x)$ and $\phi(f_1(x)) +\bar{\phi}(\lambda(x))$, respectively.   According to  Lemma \ref{agwf-},  the compositional inverse of $f(x)$ over $\gf_{p^{mn}}$ is 
	$$f^{-1}(x)=(\phi(f_1) +\bar{\phi}(\lambda))^{-1}\circ \left(\phi (x-h(g^{-1}(\Tr_m^{nm}(x))) )+\bar{\phi}(g^{-1}(\Tr_m^{nm}(x))) \right),$$  
	where $h(x)=\sum_{i=1}^kb_i(x^{t_i}+\delta)^{s_i}.$
	This completes the proof.
\end{proof}

Drawing upon Lemma \ref{th1}, it becomes evident that acquiring  the  compositional inverse of $g(x)$ over $\gf_{p^m}$ empowers us to ascertain the  compositional inverse of $f(x)$ over $\gf_{p^{nm}}.$

We list two results about the compositional inverses of linearized permutation polynomials at last. 

For a positive integer $m, $ 
define a sequence 
$$S_{-1}=0, S_0=1, S_i=b^{2^{i-1}}S_{i-1}+a^{2^{i-1}}S_{i-2},$$ 
where $1\leq i\leq m$ and $a, b \in \gf_{2^m}^*.$ Y. Zheng, Q. Wang and W. Wei \cite{zheng2019inverses} studied  the inverse of linearized polynomal of the form $x^4+bx^2+ax$ over $\gf_{2^m}.$
\begin{lemma}\label{421} \cite[Corollary 4]{zheng2019inverses}
	Let $L(x)=x^4+bx^2+ax$, where $a ,b \in \gf_{2^m}^*$ and $m>1.$ Then $L(x)$ is a permutation polynomial over $\gf_{2^m}$ if and only if $S_m+aS_{m-2}^2=1.$ Moreover, if $L(x)$ permutes $\gf_{2^m},$ the inverse of $L(x)$ over $\gf_{2^m}$ is given by  
	$$L^{-1}(x)=\sum_{i=0}^{m-1}(S_{m-2-i}^{2^{i+1}}+a^{1-2^{i+1}}S_i)x^{2^i}.$$ 
	
\end{lemma}

\begin{lemma}\label{binomial}\cite[Theorem 2.1]{baofengwu2013compositional}
	Let $L_r(x)=x^{q^r}-ax$, where $a\in \gf_{q^m}^*$, and $1\leq r\leq m-1$. Then $L_r(x)$ is a permutation polynomial over $\gf_{q^m}$ if and only if the norm $N_{q^m/q^d}(a)\neq 1,$ where $d=gcd(m ,r).$ In this case, its inverse on $\gf_{q^m}$ is 
	$$L_r^{-1}(x)=\frac{N_{q^m/q^d}(a)}{1-N_{q^m/q^d}(a)}\sum_{i=0}^{m/d-1}a^{-\frac{q^{(i+1)r}-1}{q^r-1}}x^{q^{ir}}.$$
\end{lemma}

\section{The compositional inverses of the permutation polynomials of the form $\sum_{i=1}^kb_i\left(\Tr_m^{mn}(x)+\delta\right)^{s_i}+f_1(x)$ over finite fields}
This section analyzes the compositional inverses of permutation polynomials of the form
$$f(x)=\sum_{i=1}^kb_i\left(\Tr_m^{mn}(x)+\delta\right)^{s_i}+f_1(x)$$ over $\gf_{p^{nm}}$,
where for $1 \leq i \leq k,$  $m, n, s_i$ are positive integers, $b_i\in \gf_{p^m}$, $p$ is prime, and $\delta\in \gf_{p^{nm}}.$

As immediate consequence of Lemma \ref{th1}, we get the following two results.
\begin{theorem}\label{th31}
	For  positive integers $d, n, m $ and a  prime $p,$ let $d \mid (p^{m}+1),$ $(2p)\mid n,$ and  $\delta \in \gf_{p^{mn}}.$    Assume that $f_1(x)$ is  polynomials over $\gf_{p^{mn}}$ and $\varphi(x)$ is a polynomial over $\gf_{p^m}$ satisfying the following conditions:\\ 
	(i) $\Tr_m^{mn}(x) \circ f_1(x)=\varphi(x) \circ \Tr_m^{mn}(x);$\\ 
	(ii) For any $a \in \gf_{p^m},$ $f_1(x)$ is injective on $\Tr_m^{mn}(a)^{-1};$\\
	For $1\leq i \leq k,$ $b_i \in \gf_{p^m}$ and $s_i$ are positive integers, then $$f(x)=\sum_{i=1}^kb_i\left(\Tr_m^{mn}(x)+\delta\right)^{1+s_i(p^{nm}-1)/d}+f_1(x)$$ permutes $\gf_{p^{mn}}$ if and only if $\varphi(x)$ permutes $\gf_{p^m}.$ Moreover,  if there exist two polynomials  $\phi(x)$ and  $\bar{\phi}(x)$ such that $\phi(x)\circ f_1(x)+\bar{\phi}(x)\circ \Tr_m^{nm}(x)$ permutes $\gf_{p^{nm}}$ and  $f(x)$ permutes $\gf_{p^{nm}}$,  then  the compositional inverse of $f(x)$ is 
	
	$$f^{-1}(x)=(\phi(f_1) +\bar{\phi}(\lambda))^{-1}\circ \left(\phi (x-h(\varphi^{-1}(\Tr_m^{nm}(x))) )+\bar{\phi}(\varphi^{-1}(\Tr_m^{nm}(x))) \right),$$  
	where $h(x)=\sum_{i=1}^kb_i(x+\delta)^{1+s_i(p^{nm}-1)/d}$, $\varphi^{-1}(x)$  and $(\phi(f_1(x)) +\bar{\phi}(\lambda(x)))^{-1}$ are the compositional inverses of $\varphi(x)$ and $\phi(f_1(x)) +\bar{\phi}(\lambda(x))$, respectively.
\end{theorem}
\begin{proof}
	According to Lemma \ref{th1}, $f(x)$ permutes $\gf_{p^{nm}}$ if and only if  $g(x)=\sum_{j=0}^{n-1}\sum_{i=1}^kb_i(x+\delta)^{\left({1+s_i(p^{nm}-1)/d}\right)p^{mj}}+\varphi(x)$ permutes $\gf_{p^m}.$
	Since $2\mid n$ and $d \mid (p^{m}+1),$ and since moreover $x^{p^{nm}}=x$ for any $x \in \gf_{p^{nm}},$ 
	we have $$(x^{\frac{p^{nm}-1}{d}})^{p^{2lm}}=x^{\frac{p^{nm}-1}{d}} \,\, \text{and} \,\, (x^{\frac{p^{nm}-1}{d}})^{p^{(2l+1)m}}=x^{\frac{p^m(p^{nm}-1)}{d}} $$ for any  $l \in\mathbb{Z}^+.$ This implies   
	\begin{align}\label{1eqht31}
		g(x)=&\,\sum_{j=0}^{n-1}\sum_{i=1}^kb_i(x+\delta)^{\left({1+s_i(p^{nm}-1)/d}\right)p^{mj}}+\varphi(x)\nonumber\\
		=&\,\sum_{i=1}^kb_i(x+\delta)^{\frac{s_i(p^{nm}-1)}{d}}\sum_{l=0}^{\frac{n}{2}-1}(x+\delta^{p^{2lm}})\nonumber\\
		&\,+\sum_{i=1}^kb_i(x+\delta)^{\frac{s_ip^m(p^{nm}-1)}{d}}\sum_{l=0}^{\frac{n}{2}-1}(x+\delta^{p^{(2l+1)m}})+\varphi(x).
	\end{align}
	Note that if $(2p) \mid n,$ then 
	$$\sum_{l=0}^{\frac{n}{2}-1}(x+\delta^{p^{2lm}})=
	p\sum_{l=0}^{\frac{n}{2p}-1}(x+\delta^{p^{2lm}})=0,$$ and
	$$\sum_{l=0}^{\frac{n}{2}-1}(x+\delta^{p^{(2l+1)m}})=
	p\sum_{l=0}^{\frac{n}{2p}-1}(x+\delta^{p^{(2l+1)m}})=0.$$
	Therefore, by Eq.\eqref{1eqht31}, $g(x)=\varphi(x).$ This implies that $f(x)$ permutes $\gf_{p^{mn}}$ if and only if $\varphi(x)$ permutes $\gf_{p^m}.$ 
	We  immediately  conclude the second result by Lemma \ref{th1}.
	This completes the proof.
\end{proof}
\begin{example}
	In Theorem \ref{th31}, for the case where $k=1,$ $b_1=1,$ $p=2$, $n=4,$  $d=q+1$,
	and $f_1(x)=x, $ we obtain Theorem 3 in \cite{zeng2015permutation}. Furthermore, by setting  $\phi(x)=x$ and $\bar{\phi}(x)=0$ in Lemma \ref{th1}, we establish 
	$\phi(x)\circ f_1(x)+\bar{\phi}(x)\circ \Tr_m^{nm}(x)=x.$ Hence,  the compositional inverse of $f(x)=(\Tr_m^{4m}(x)+\delta)^{1+s_1(2^{4m}-1)/(2^m+1)}+x$ over $\gf_{2^{4m}}$ is given by 
	$$f^{-1}(x)=x+(\Tr_m^{4m}(x)+\delta)^{1+s_1(2^{4m}-1)/(2^m+1)}=f(x),$$ indicating  that $f(x)$ is an involution over $\gf_{2^{4m}}.$
\end{example}

\begin{theorem}\label{th21}
	For  positive integers $d, n, m $ and a  prime $p,$ let $d \mid (p^{m}-1),$ $\delta \in \gf_{p^{mn}} $ with $\Tr_m^{nm}(\delta)=0$, and  $p$ be a divisor of $n.$   Assume that $f_1(x)$ is  polynomials over $\gf_{p^{mn}}$ and $\varphi(x)$ is a polynomial over $\gf_{p^m}$ satisfying the following conditions:\\ 
	(i) $\Tr_m^{mn}(x) \circ f_1(x)=\varphi(x) \circ \Tr_m^{mn}(x);$\\ 
	(ii) For any $a \in \gf_{p^m},$ $f_1(x)$ is injective on $\Tr_m^{mn}(a)^{-1};$\\
	For $1\leq i \leq k,$ $b_i \in \gf_{p^m}$ and $s_i$ are positive integers, then $$f(x)=\sum_{i=1}^kb_i\left(\Tr_m^{mn}(x)+\delta\right)^{1+s_i(p^{nm}-1)/d}+f_1(x)$$ permutes $\gf_{p^{mn}}$ if and only if $\varphi(x)$ permutes $\gf_{p^m}.$ Moreover,  if there exist two polynomials  $\phi(x)$ and  $\bar{\phi}(x)$ such that $\phi(x)\circ f_1(x)+\bar{\phi}(x)\circ \Tr_m^{nm}(x)$ permutes $\gf_{p^{nm}}$ and  $f(x)$ permutes $\gf_{p^{nm}}$,  then  the compositional inverse of $f(x)$  over $\gf_{p^{nm}}$ is 
	
	$$f^{-1}(x)=(\phi(f_1) +\bar{\phi}(\lambda))^{-1}\circ \left(\phi (x-h(\varphi^{-1}(\Tr_m^{nm}(x))) )+\bar{\phi}(\varphi^{-1}(\Tr_m^{nm}(x))) \right),$$  
	where $h(x)=\sum_{i=1}^kb_i(x+\delta)^{1+s_i(p^{nm}-1)/d}$, $\varphi^{-1}(x)$  and $(\phi(f_1(x)) +\bar{\phi}(\lambda(x)))^{-1}$ are the compositional inverses of $\varphi(x)$ and $\phi(f_1(x)) +\bar{\phi}(\lambda(x))$, respectively.
	
\end{theorem}

\begin{proof}
It follows from  Lemma \ref{th1} that $f(x)$ permutes $\gf_{p^{nm}}$ if and only if  $g(x)=\sum_{j=0}^{n-1}\sum_{i=1}^kb_i(x+\delta)^{\left({1+s_i(p^{nm}-1)/d}\right)p^{mj}}+\varphi(x)$ permutes $\gf_{p^m}.$
	Moreover, because $d \mid (p^{m}-1)$ and $x^{p^{nm}}=x$ for any $x\in \gf_{p^{nm}}$, we have $x^{\frac{p^{nm}-1}{d}\cdot p^{mj}}=x^{\frac{p^{nm}-1}{d}}$ for any non-negative integer $j$.  Therefore, we have 
	\begin{align*}
		g(x)=&\,\sum_{j=0}^{n-1}\sum_{i=1}^kb_i(x+\delta)^{\left(s_i(p^{nm}-1)/d+1\right)p^{mj}}+\varphi(x)\\
		=&\, \sum_{i=1}^kb_i(x+\delta)^{s_i(p^{nm}-1)/d}\sum_{j=0}^{n-1}(x+\delta)^{p^{mj}}+\varphi(x)\\
		=&\, \sum_{i=1}^kb_i(x+\delta)^{s_i(p^{nm}-1)/d}(nx+\Tr_m^{nm}(\delta))+\varphi(x)\\
		=&\,\varphi(x).
	\end{align*}
	Hence, we get the desired result by Lemma \ref{th1}. 
\end{proof}

We provide two examples of permutation polynomials that were extensively studied in a previous work \cite{wu2022further} in corollary 1 regarding their permutation properties. Our focus will be on examining their compositional inverses. 

\begin{example}
Considering the notations defined in 	
	 Theorem \ref{th21}, with  $f_1(x)=\Tr_m^{nm}(x)+x$,  the polynomial  $$\sum_{i=1}^kb_i\left(\Tr_m^{mn}(x)+\delta\right)^{s_i(p^{nm}-1)/d+1}+\Tr_m^{nm}(x)+x$$  permutes $\gf_{p^{nm}}.$ 
Furthermore, utilizing $\phi(x)=x$ and $\bar{\phi}(x)=-x$ in Lemma \ref{th1}, we establish the equation
	$\phi(x)\circ f_1(x)+\bar{\phi}(x)\circ \Tr_m^{nm}(x)=x.$ Consequently,  the compositional inverse of $f(x)$ over $\gf_{p^{mn}}$ is given by 
	$$f^{-1}(x)=-\Tr_m^{mn}(x)+x-\sum_{i=1}^kb_i\left(\Tr_m^{mn}(x)+\delta\right)^{s_i(p^{nm}-1)/d+1}$$ according to Theorem \ref{th21}.	 
\end{example} 

\begin{example}
Under the definitions provided in  Theorem \ref{th21}, assuming 
	 $f_1(x)=x$,  the polynomial  $$\sum_{i=1}^kb_i\left(\Tr_m^{mn}(x)+\delta\right)^{s_i(p^{nm}-1)/d+1}+x$$  permutes $\gf_{p^{nm}}.$ 
Additionally, employing  $\phi(x)=x$ and $\bar{\phi}(x)=0$ in Lemma \ref{th1}, we derive the equation 
	$\phi(x)\circ f_1(x)+\bar{\phi}(x)\circ \Tr_m^{nm}(x)=x.$ Consequently,  the compositional inverse of $f(x)$ over $\gf_{p^{mn}}$ can be expressed as  
	$$f^{-1}(x)=x-\sum_{i=1}^kb_i\left(\Tr_m^{mn}(x)+\delta\right)^{s_i(p^{nm}-1)/d+1}$$ by Theorem \ref{th21}.		
\end{example}

Next, we will consider the  compositional inverses of permutation polynomials of the form
$$f(x)=\left(\Tr_m^{mn}(x)+\delta\right)^{s_i}+f_1(x)$$ over $\gf_{p^{nm}}$ with $\delta \in \gf_{p^{nm}}$ in the following four theorems.

The authors in \cite{wu2022further} explored the permutation properties exhibited by the polynomial $(\Tr_m^{3m}(x)+\delta)^{q^2+q+2}+\Tr_{m}^{3m}(x)^{2^l}+x^{2^l}$ over $\gf_{q^3}$, where $q=2^m$ with $m$ being a positive integer,  $\delta \in \gf_{q^3}$, and $l$ is a non-negative integer. In the subsequent theorem, we aim to  determine the compositional inverse of permutation polynomial in this particular form. 

\begin{theorem}
	For a positive integer $m$ and a non-negative integer $l$, let $q=2^m,$ $\delta \in \gf_{q^3},$ $A=\Tr_m^{3m}(\delta^{2+q}+\delta^{2+q^2}),$  $B=\Tr_m^{3m}(\delta^2+\delta^{1+q}), $ and   $D=\delta^{q^2+q+1}\Tr_m^{3m}(\delta).$   Assume that  $$f(x)=(\Tr_m^{3m}(x)+\delta)^{q^2+q+2}+\Tr_{m}^{3m}(x)^{2^l}+x^{2^l}$$ permutes $ \gf_{q^3} .$ \\
	If $A=0$ and $B=0,$ then the compositional inverse of $f(x)$ over $\gf_{q^3}$ is 
	$$f^{-1}(x)=x^{q^3/2^l}+\left((\Tr_{m}^{3m}(x)+D)^{q/4}+\delta\right)^{(q^2+q+2)q^3/2^l}+(\Tr_{m}^{3m}(x)+D)^{q/4}.$$
	If $B=0$ and $A\neq0 $
	is not a cubic of some element in $\gf_q$, then the compositional inverse of $f(x)$ over $\gf_{q^3}$ is{\footnotesize
		\begin{align*}
			f^{-1}(x)
			=&\, \left(\delta+\frac{N_{q^m/q^d}(A)}{1+N_{q^m/q^d}(A)}\sum_{i=0}^{m/d-1}A^{-\frac{4^{i+1}-1}{3}}(\Tr_{m}^{3m}(x)+D)^{4^i}\right)^{(q^2+q+2)q^3/2^l}\\
			&\,+x^{q^3/2^l}+\frac{N_{q^m/q^d}(A)}{1+N_{q^m/q^d}(A)}\sum_{i=0}^{m/d-1}A^{-\frac{4^{i+1}-1}{3}}(\Tr_{m}^{3m}(x)+D)^{4^i},
		\end{align*}
	} where  $d=\gcd(m, 2).$\\
	If $AB\neq0,$  then the compositional inverse of $f(x)$ over $\gf_{q^3}$ is 
	{\footnotesize
		\begin{align*}
			f^{-1}(x)
			=&\, \left(\delta+\sum_{i=0}^{m-1}(S_{m-2-i}^{2^{i+1}}+A^{1-2^{i+1}}S_i)(\Tr_m^{nm}(x)+D)^{2^i}\right)^{(q^2+q+2)q^3/2^l}\\
			&\,+x^{q^3/2^l}+\sum_{i=0}^{m-1}(S_{m-2-i}^{2^{i+1}}+A^{1-2^{i+1}}S_i)(\Tr_m^{nm}(x)+D)^{2^i},
		\end{align*}
	} where $S_i$ is a sequence with $S_{-1}=0, S_0=1, S_i=B^{2^{i-1}}S_{i-1}+A^{2^{i-1}}S_{i-2}.$
\end{theorem}

\begin{proof}
	
	Since \begin{align*}
		\Tr_m^{3m}(x)\circ \left(\Tr_{m}^{3m}(x)^{2^l}+x^{2^l}\right)=&\,\sum^{2}_{i=0}(\Tr_m^{3m}(x)^{q^i})^{2^l}+\Tr_m^{3m}(x)^{2^l}\\
		=&\,\sum^{2}_{i=0}\Tr_m^{3m}(x)^{2^l}+\Tr_m^{3m}(x)^{2^l}=0,
	\end{align*}
	it follows from Lemma \ref{th1} that if $f(x)$ permutes $\gf_{q^3}$, then $g(x)=\sum_{j=0}^{2}(x+\delta)^{q^j(q^2+q+2)}$
	permutes $\gf_{q}.$
	Furthermore, since $q^2+q+2=(q^2+q+1)+1$, we have 
	\begin{align}\label{1eq2^l}
		g(x)=&\,\sum_{j=0}^{2}(x+\delta)^{q^j\left((q^2+q+1)+1
			\right)}\nonumber\\
		=&\,(x+\delta)^{q^2+q+1}(x+\Tr_m^{3m}(\delta))\nonumber\\
		=&\,(x^3+\Tr_m^{3m}(\delta)x^2+\Tr_m^{3m}(\delta^{1+q})x+\delta^{q^2+q+1})(x+\Tr_m^{3m}(\delta))\nonumber\\
		=&\,x^4+\Tr_m^{3m}(\delta^2+\delta^{1+q})x^2+\Tr_m^{3m}(\delta^{2+q}+\delta^{2+q^2})x+\delta^{q^2+q+1}\Tr_m^{3m}(\delta)\nonumber\\
		=&\, x^4+Bx^2+Ax+D.
	\end{align}
By setting  $\phi(x)=x$ and $\bar{\phi}(x)=x^{2^l}$ in Lemma \ref{th1}, we have  
	that 
	\begin{equation}\label{0eq2^l}
		\phi(\Tr_{m}^{3m}(x)^{2^l}+x^{2^l})+\bar{\phi}(\Tr_{m}^{3m}(x))=x^{2^l}\end{equation}
	permutes $\gf_{2^{3m}}.$\\
	If $A=0$ and $B=0,$ then by Eq. \eqref{1eq2^l},  $g(x)=x^4+D=(x+D)\circ x^4,$ and so 
	\begin{equation}\label{2eq2^l}
		g^{-1}(x)=x^{q/4}\circ (x+D)=(x+D)^{q/4}.
	\end{equation}
	It follows from Lemma \ref{th1}, Eqs. \eqref{0eq2^l} and  \eqref{2eq2^l} that the compositional inverse of $f(x)$ over $\gf_{q^3}$ is 
	\begin{align*}
		f^{-1}(x)=&\, x^{q^3/2^l} \circ \left(\left(x+(g^{-1}(\Tr_{m}^{3m}(x))+\delta)^{q^2+q+2}\right)+x^{2^l}\circ g^{-1}(\Tr_{m}^{3m}(x)) \right)\\
		=&\, x^{q^3/2^l}+\left((\Tr_{m}^{3m}(x)+D)^{q/4}+\delta\right)^{(q^2+q+2)q^3/2^l}+(\Tr_{m}^{3m}(x)+D)^{q/4}.
	\end{align*}
	If $B=0$ and $A\neq0 $
	is not a cubic of some element in $\gf_q$, then by Eq. \eqref{1eq2^l}, 
	\begin{align*}
		g(x)=&\, 
		x^4+Ax+D =\left(x+D\right) \circ\left(x^4+Ax\right), 
	\end{align*}
	and so, by Lemma \ref{binomial},
	\begin{align}\label{3eq2^l}
		g^{-1}(x)=&\,\left(\frac{N_{q^m/q^d}(A)}{1+N_{q^m/q^d}(A)}\sum_{i=0}^{m/d-1}A^{-\frac{4^{i+1}-1}{3}}x^{4^i}\right) \circ (x+D)\nonumber\\
		=&\, \frac{N_{q^m/q^d}(A)}{1+N_{q^m/q^d}(A)}\sum_{i=0}^{m/d-1}A^{-\frac{4^{i+1}-1}{3}}(x+D)^{4^i},
	\end{align}
	where $d=\gcd(m, 2).$
	It implies by Lemma \ref{th1}, Eqs. \eqref{0eq2^l} and \eqref{3eq2^l} that the compositional inverse of $f(x)$ over $\gf_{q^3}$ is 
	{\footnotesize
		\begin{align*}
			f^{-1}(x)
			=&\, \left(\delta+\frac{N_{q^m/q^d}(A)}{1+N_{q^m/q^d}(A)}\sum_{i=0}^{m/d-1}A^{-\frac{4^{i+1}-1}{3}}(\Tr_{m}^{3m}(x)+D)^{4^i}\right)^{(q^2+q+2)q^3/2^l}\\
			&\,+x^{q^3/2^l}+\frac{N_{q^m/q^d}(A)}{1+N_{q^m/q^d}(A)}\sum_{i=0}^{m/d-1}A^{-\frac{4^{i+1}-1}{3}}(\Tr_{m}^{3m}(x)+D)^{4^i}.
		\end{align*}
	}
	If $AB\neq0,$
	then by Eq. \eqref{1eq2^l}, 
	\begin{align*}
		g(x)=&\, 
		\left(x+D\right) \circ\left(x^4+Bx^2+Ax\right), 
	\end{align*}
	and so, by Lemma \ref{binomial},
	\begin{align}\label{4eq2^l}
		g^{-1}(x)=&\,\sum_{i=0}^{m-1}(S_{m-2-i}^{2^{i+1}}+A^{1-2^{i+1}}S_i)x^{2^i}\circ (x+D)\nonumber\\
		=&\,\sum_{i=0}^{m-1}(S_{m-2-i}^{2^{i+1}}+A^{1-2^{i+1}}S_i)(x+D)^{2^i}, 
	\end{align}
	where $S_i$ is a sequence with $S_{-1}=0, S_0=1, S_i=B^{2^{i-1}}S_{i-1}+A^{2^{i-1}}S_{i-2}.$\\
	It implies by Lemma \ref{th1}, Eqs. \eqref{0eq2^l} and \eqref{4eq2^l} that the compositional inverse of $f(x)$ over $\gf_{q^3}$ is 
	{\footnotesize
		\begin{align*}
			f^{-1}(x)
			=&\, \left(\delta+\sum_{i=0}^{m-1}(S_{m-2-i}^{2^{i+1}}+A^{1-2^{i+1}}S_i)(\Tr_m^{nm}(x)+D)^{2^i}\right)^{(q^2+q+2)q^3/2^l}\\
			&\,+x^{q^3/2^l}+\sum_{i=0}^{m-1}(S_{m-2-i}^{2^{i+1}}+A^{1-2^{i+1}}S_i)(\Tr_m^{nm}(x)+D)^{2^i}.
		\end{align*}
	} This completes the proof. 
\end{proof}

In \cite{wu2022further}, the permutation properties of $(\Tr_m^{3m}(x)+\delta)^{q^2+q+2}+x$ over $\gf_{q^3}$ were investigated, where  $q=2^m$ with $m$ being  a positive integer and   $\delta \in \gf_{q^3}$. We will focus on  determining the compositional inverse of permutation polynomial in this specific form in the following theorem. 

\begin{theorem}\label{th2^0}
	For a positive integer $m$, let $q=2^m,$  $\delta \in \gf_{q^3},$ $A=\Tr_m^{3m}(\delta^{2+q}+\delta^{2+q^2}+1),$   $B=\Tr_m^{3m}(\delta^2+\delta^{1+q}), $ and  $D=\delta^{q^2+q+1}\Tr_m^{3m}(\delta).$  Assume that  $$f(x)=(\Tr_m^{3m}(x)+\delta)^{q^2+q+2}+x$$ permutes $ \gf_{q^3} .$ \\
	If $A=0$ and $B=0,$ then the compositional inverse of $f(x)$ over $\gf_{q^3}$ is 	\begin{align*}
		f^{-1}(x)=&\, x+\left((\Tr_m^{nm}(x)+D)^{q/4}+\delta\right)^{q^2+q+2} .
	\end{align*}
	If $B=0$ and $A\neq0 $
	is not a cubic of some element in $\gf_q$, then the compositional inverse of $f(x)$ over $\gf_{q^3}$ is 	\begin{align*}
		f^{-1}(x)	
		=&\,x+ \left(\delta+\frac{N_{q^m/q^d}(A)}{1+N_{q^m/q^d}(A)}\sum_{i=0}^{m/d-1}A^{-\frac{4^{i+1}-1}{3}}(\Tr_m^{nm}(x)+D)^{4^i}\right)^{q^2+q+2},
	\end{align*}		
	where  $d=\gcd(m, 2).$\\
	If $AB\neq0,$  then the compositional inverse of $f(x)$ over $\gf_{q^3}$ is 	\begin{align*}			
		f^{-1}(x)		
		=&\,x+ \left(\delta+\sum_{i=0}^{m-1}(S_{m-2-i}^{2^{i+1}}+A^{1-2^{i+1}}S_i)(\Tr_m^{nm}(x)+D)^{2^i}\right)^{q^2+q+2},		
	\end{align*}
	where $S_i$ is a sequence with $S_{-1}=0, S_0=1, S_i=B^{2^{i-1}}S_{i-1}+A^{2^{i-1}}S_{i-2}.$
\end{theorem}

\begin{proof}
	Since $
	\Tr_m^{3m}(x)\circ x= x \circ \Tr_m^{3m}(x), $ 
	it follows from Lemma \ref{th1} that if $f(x)$ permutes $\gf_{q^3}$, then $g(x)=\sum_{j=0}^{2}(x+\delta)^{q^j(q^2+q+2)}+x$
	permutes $\gf_{q}.$
	Furthermore, since $q^2+q+2=(q^2+q+1)+1$, we have 
	\begin{align}\label{1eq2^0}
		g(x)=&\,\sum_{j=0}^{2}(x+\delta)^{q^j\left((q^2+q+1)+1
			\right)}+x\nonumber\\
		=&\,(x+\delta)^{q^2+q+1}(x+\Tr_m^{3m}(\delta))+x\nonumber\\
		=&\,(x^3+\Tr_m^{3m}(\delta)x^2+\Tr_m^{3m}(\delta^{1+q})x+\delta^{q^2+q+1})(x+\Tr_m^{3m}(\delta))+x\nonumber\\
		=&\,x^4+\Tr_m^{3m}(\delta^2+\delta^{1+q})x^2+\Tr_m^{3m}(\delta^{2+q}+\delta^{2+q^2}+1)x+\delta^{q^2+q+1}\Tr_m^{3m}(\delta)\nonumber\\
		=&\, x^4+Bx^2+Ax+D.
	\end{align}

	By setting  $\phi(x)=x$ and $\bar{\phi}(x)=0$ in Lemma \ref{th1}, we have 
	that 
	\begin{equation}\label{0eq2^0}
		\phi(x)+\bar{\phi}(\Tr_{m}^{3m}(x))=x\end{equation}
	permutes $\gf_{q^3}.$\\
	If $A=0$ and $B=0,$ then by Eq. \eqref{1eq2^0},  $g(x)=x^4+D=(x+D)\circ x^4,$ and so 
	\begin{equation}\label{2eq2^0}
		g^{-1}(x)=x^{q/4}\circ (x+D)=(x+D)^{q/4}.
	\end{equation}
	It follows from Lemma \ref{th1}, Eqs. \eqref{0eq2^0} and  \eqref{2eq2^0} that the compositional inverse of $f(x)$ over $\gf_{q^3}$ is 
	\begin{align*}
		f^{-1}(x)
		=&\, x+\left((\Tr_m^{nm}(x)+D)^{q/4}+\delta\right)^{q^2+q+2} .
	\end{align*}
	If $B=0$ and $A\neq0 $
	is not a cubic of some element in $\gf_q$, then by Eq. \eqref{1eq2^0}, 
	\begin{align*}
		g(x)=&\, 
		x^4+Ax+D =\left(x+D\right) \circ\left(x^4+Ax\right), 
	\end{align*}
	and so, by Lemma \ref{binomial},
	\begin{align}\label{3eq2^0}
		g^{-1}(x)=&\,\left(\frac{N_{q^m/q^d}(A)}{1+N_{q^m/q^d}(A)}\sum_{i=0}^{m/d-1}A^{-\frac{4^{i+1}-1}{3}}x^{4^i}\right) \circ (x+D)\nonumber\\
		=&\, \frac{N_{q^m/q^d}(A)}{1+N_{q^m/q^d}(A)}\sum_{i=0}^{m/d-1}A^{-\frac{4^{i+1}-1}{3}}(x+D)^{4^i},
	\end{align}
	where $d=\gcd(m, 2).$
	It implies by Lemma \ref{th1}, Eqs. \eqref{0eq2^0} and \eqref{3eq2^0} that the compositional inverse of $f(x)$ over $\gf_{q^3}$ is 
	\begin{align*}
		f^{-1}(x)
		=&\,x+ \left(\delta+\frac{N_{q^m/q^d}(A)}{1+N_{q^m/q^d}(A)}\sum_{i=0}^{m/d-1}A^{-\frac{4^{i+1}-1}{3}}(\Tr_m^{nm}(x)+D)^{4^i}\right)^{q^2+q+2} .
	\end{align*}	
	If $AB\neq0,$
	then by Eq. \eqref{1eq2^l}, 
	\begin{align*}
		g(x)=&\, 
		\left(x+D\right) \circ\left(x^4+Bx^2+Ax\right), 
	\end{align*}
	and so, by Lemma \ref{binomial},
	\begin{align}\label{4eq2^0}
		g^{-1}(x)=&\,\sum_{i=0}^{m-1}(S_{m-2-i}^{2^{i+1}}+A^{1-2^{i+1}}S_i)x^{2^i} \circ (x+D)\nonumber\\
		=&\,\sum_{i=0}^{m-1}(S_{m-2-i}^{2^{i+1}}+A^{1-2^{i+1}}S_i)(x+D)^{2^i}, 
	\end{align}
	where $S_i$ is a sequence with $S_{-1}=0, S_0=1, S_i=B^{2^{i-1}}S_{i-1}+A^{2^{i-1}}S_{i-2}.$\\
	It implies by Lemma \ref{th1}, Eqs. \eqref{0eq2^0} and \eqref{4eq2^0} that the compositional inverse of $f(x)$ over $\gf_{q^3}$ is 
	\begin{align*}			
		f^{-1}(x)
		=&\,x+ \left(\delta+\sum_{i=0}^{m-1}(S_{m-2-i}^{2^{i+1}}+A^{1-2^{i+1}}S_i)(\Tr_m^{nm}(x)+D)^{2^i}\right)^{q^2+q+2}.		
	\end{align*}
	We are done. 
\end{proof}

In their study highlighted by \cite{wu2022further}, analysis was conducted on the  permutation properties exhibited by the polynomial  $(\Tr_m^{3m}(x)+\delta)^{q^2+q+2}+x^2$ over $\gf_{q^3}$. Here  $q=2^m$ with $m$ as a positive integer  and  $\delta \in \gf_{q^3}$.  Subsequently, the forthcoming theorem will focus on  determining the compositional inverse of permutation polynomial in this specific form.  As the proof closely resembles that of  Theorem \ref{th2^0}, the primary distinction lies in the function   $	g(x)=\,\sum_{j=0}^{2}(x+\delta)^{q^j\left((q^2+q+1)+1
	\right)}+x^2.$  Consequently, we omit the specific details. 

\begin{theorem}\label{th2^1}
	For a positive integer $m$, let $q=2^m,$  $\delta \in \gf_{q^3},$ $A=\Tr_m^{3m}(\delta^{2+q}+\delta^{2+q^2}),$  $B=\Tr_m^{3m}(\delta^2+\delta^{1+q}+1)$  and $D=\delta^{q^2+q+1}\Tr_m^{3m}(\delta).$ Assume that  $$f(x)=(\Tr_m^{3m}(x)+\delta)^{q^2+q+2}+x^2$$ permutes $ \gf_{q^3} .$ \\
	If $A=0$ and $B=0,$ then the compositional inverse of $f(x)$ over $\gf_{q^3}$ is 	\begin{align*}
		f^{-1}(x)=&\, x+\left((\Tr_m^{nm}(x)+D)^{q/4}+\delta\right)^{q^2+q+2} .
	\end{align*}
	If $B=0$ and $A\neq0 $
	is not a cubic of some element in $\gf_q$, then the compositional inverse of $f(x)$ over $\gf_{q^3}$ is 	\begin{align*}
		f^{-1}(x)	
		=&\,x+ \left(\delta+\frac{N_{q^m/q^d}(A)}{1+N_{q^m/q^d}(A)}\sum_{i=0}^{m/d-1}A^{-\frac{4^{i+1}-1}{3}}(\Tr_m^{nm}(x)+D)^{4^i}\right)^{q^2+q+2},
	\end{align*}		
	where  $d=\gcd(m, 2).$\\
	If $AB\neq0,$  then the compositional inverse of $f(x)$ over $\gf_{q^3}$ is 	\begin{align*}			
		f^{-1}(x)		
		=&\,x+ \left(\delta+\sum_{i=0}^{m-1}(S_{m-2-i}^{2^{i+1}}+A^{1-2^{i+1}}S_i)(\Tr_m^{nm}(x)+D)^{2^i}\right)^{q^2+q+2},		
	\end{align*}
	where $S_i$ is a sequence with $S_{-1}=0, S_0=1, S_i=B^{2^{i-1}}S_{i-1}+A^{2^{i-1}}S_{i-2}.$
\end{theorem}

The study conducted in \cite{wu2022further} explored  the permutation properties exhibited by $(\Tr_m^{3m}(x)+\delta)^{q^2+q+2}+x^4$ over $\gf_{q^3}$, where $q=2^m$ with a positive integer $m$ and $\delta \in \gf_{q^3} $.  We will determine the compositional inverse of permutation polynomial in this specific form in the  following theorem. 

\begin{theorem}
	Let $q=2^m$ with a positive integer $m.$  For any  $\delta \in \gf_{q^3}$, assume that 
	$$f(x)=(\Tr_m^{3m}(x)+\delta)^{q^2+q+2}+x^4$$
	permutes $\gf_{q^3}.$ \\
	If  $\Tr_m^{3m}(\delta^2+\delta^{1+q})=0$ and $\Tr_m^{3m}(\delta^2(\delta^{q}+\delta^{q^2}))\neq0, $ then the compositional inverse of $f(x)$ over $\gf_{q^3}$ is 
	$$f^{-1}(x)=x+\left(\delta+\Tr_m^{3m}(\delta^{2+q}+\delta^{2+q^2})^{-1}\left(\Tr_m^{nm}(x)+\delta^{1+q+q^2}\Tr_m^{3m}(\delta)\right)\right)^{q^2+q+2}.$$	
	If $\Tr_m^{3m}(\delta^2+\delta^{1+q})\neq0$ and $\Tr_m^{3m}(\delta^2(\delta^{q}+\delta^{q^2}))=0, $  then the compositional inverse of $f(x)$ over $\gf_{q^3}$ is 
	$$f^{-1}(x)=x+\left(\delta+\left(\Tr_m^{3m}(\delta^2+\delta^{1+q})^{-1}(\Tr_m^{nm}(x)+\delta^{1+q+q^2}\Tr_m^{3m}(\delta))\right)^{q/2}\right)^{q^2+q+2}.$$
\end{theorem}
\begin{proof}
	Since $\Tr_m^{3m}(x)$ and $x^4$ are linearized polynomials, we have 
	$\Tr_m^{3m}(x)\circ x^4=x^4 \circ\Tr_m^{3m}(x).$
	It follows from Lemma \ref{th1} that if $f(x)$ permutes $\gf_{q^3},$ then 
	$$g(x)=(x+\delta)^{\frac{q^3-1}{q-1}}(x+\Tr_m^{3m}(\delta))+x^4$$
	permutes $\gf_{q}.$
	Moreover,
	\begin{align}\label{1eq2^2}
		g(x)=&\,(x+\delta)^{1+q+q^2}(x+\Tr_m^{3m}(\delta))+x^4\nonumber\\
		=&\,(x^3+\Tr_m^{3m}(\delta)x^2+\Tr_m^{3m}(\delta^{1+q})x+\delta^{1+q+q^2})(x+\Tr_m^{3m}(\delta))+x^4\nonumber\\
		=&\,\Tr_m^{3m}(\delta^2+\delta^{1+q})x^2+\Tr_m^{3m}(\delta^{2+q}+\delta^{2+q^2})x+\delta^{1+q+q^2}\Tr_m^{3m}(\delta).
	\end{align}
	Taking $\phi(x)=x$ and $\bar{\phi}(x)=0$ in Lemma \ref{th1}, we have 
	that 
	\begin{equation}\label{0eq2^2}
		\phi(x)+\bar{\phi}(\Tr_{m}^{3m}(x))=x\end{equation}
	permutes $\gf_{2^{3m}}.$\\
	If  $\Tr_m^{3m}(\delta^2+\delta^{1+q})=0$ and $\Tr_m^{3m}(\delta^2(\delta^{q}+\delta^{q^2}))\neq0, $ then by Eq. \eqref{1eq2^2}, $g(x)=\Tr_m^{3m}(\delta^{2+q}+\delta^{2+q^2})x+\delta^{1+q+q^2}\Tr_m^{3m}(\delta),$ and so 
	\begin{equation}\label{2eq2^2}
		g^{-1}(x)=\Tr_m^{3m}(\delta^{2+q}+\delta^{2+q^2})^{-1}\left(x+\delta^{1+q+q^2}\Tr_m^{3m}(\delta)\right).
	\end{equation}
	Therefore, it implies by Lemma \ref{th1},  Eqs. \eqref{2eq2^2} and \eqref{0eq2^2}, that the compositional inverse of $f(x)$ over $\gf_{q^3}$ is  $$f^{-1}(x)=x+\left(\delta+\Tr_m^{3m}(\delta^{2+q}+\delta^{2+q^2})^{-1}\left(\Tr_m^{nm}(x)+\delta^{1+q+q^2}\Tr_m^{3m}(\delta)\right)\right)^{q^2+q+2}.$$
	If $\Tr_m^{3m}(\delta^2+\delta^{1+q})\neq0$ and $\Tr_m^{3m}(\delta^2(\delta^{q}+\delta^{q^2}))=0, $  then then by Eq. \eqref{1eq2^2}, 
	\begin{align*}
		g(x)=&\,	\Tr_m^{3m}(\delta^2+\delta^{1+q})x^2+\delta^{1+q+q^2}\Tr_m^{3m}(\delta)
		\\=&\,\left(\Tr_m^{3m}(\delta^2+\delta^{1+q})x+\delta^{1+q+q^2}\Tr_m^{3m}(\delta)\right)\circ x^2,
	\end{align*}
	and so,
	\begin{align}\label{3eq2^2}
		g^{-1}(x)=&\, x^{q/2}\circ \left(\Tr_m^{3m}(\delta^2+\delta^{1+q})^{-1}(x+\delta^{1+q+q^2}\Tr_m^{3m}(\delta))\right)\nonumber\\
		=&\, \left(\Tr_m^{3m}(\delta^2+\delta^{1+q})^{-1}(x+\delta^{1+q+q^2}\Tr_m^{3m}(\delta))\right)^{q/2}.	
	\end{align}
	Consequently, it follows from Lemma \ref{th1},  Eqs. \eqref{0eq2^2} and  \eqref{3eq2^2}, that the compositional inverse of $f(x)$ over $\gf_{q^3}$ is 
	$$f^{-1}(x)=x+\left(\delta+\left(\Tr_m^{3m}(\delta^2+\delta^{1+q})^{-1}(\Tr_m^{nm}(x)+\delta^{1+q+q^2}\Tr_m^{3m}(\delta))\right)^{q/2}\right)^{q^2+q+2}.$$
	We complete the proof.
\end{proof}

\section{The compositional inverses of the permutation polynomials of the form $\sum_{i=1}^kb_i\left(\Tr_m^{2m}(x)^{t_i}+\delta\right)^{s_i}+x$ over $\gf_{2^{2m}}$}
This section analyzes the compositional inverses of permutation polynomials of the form
$$f(x)=\sum_{i=1}^kb_i\left(\Tr_m^{2m}(x)^{t_i}+\delta\right)^{s_i}+x$$ over $\gf_{2^{2m}}$ ,
where for $1 \leq i \leq k,$ , $m, s_i$ are positive integers, $b_i\in \gf_{2^m}$ and $\delta\in \gf_{2^{2m}}.$

In \cite{li2020some}, a proposition was presented regarding the  permutation property of the polynomial
$x+\left(\Tr_m^{2m}(x)^{(2^m+1)/3}+\delta\right)^{2^{m-1}+1}$ over $\gf_{2^{2m}}$, where $\delta \in \gf_{2^{2m}}$ and $m$ is odd.  We aim to  investigate the compositional inverse of this calss of permutation polynomial in the following theorem. 
\begin{theorem}
	Let	 $\delta \in \gf_{2^{2m}}$ and $m$ be odd. Then the compositional inverse of permutation polynomial 
	$$f(x)=x+\left(\Tr_m^{2m}(x)^{(2^m+1)/3}+\delta\right)^{2^{m-1}+1}$$
	over $\gf_{2^{2m}}$ is 	$$f^{-1}(x)=x+\left((\Tr_m^{2m}(x)^2+\delta^3+\delta^{3\cdot 2^m})^{(2^{m+1}-1)/3}+\delta^{2^m}\right)^{2^{m-1}+1}.$$
\end{theorem}

\begin{proof}
	It follows from Lemma \ref{th1} that if $f(x)$ permutes $\gf_{2^{2^m}}$, then $g(x)=x+\left(x^{(2^m+1)/3}+\delta\right)^{2^{m-1}+1}+\left(x^{(2^m+1)/3}+\delta\right)^{2^{2m-1}+2^m}$ is a permutation polynomial over $\gf_{2^m}.$ 
	Since $(2^m+1)/3\cdot (3\cdot2^{m-1})\equiv1 \pmod{2^m-1}$, we have 
	\begin{align*}
		x^2 \circ g(x) \circ x^{3\cdot2^{m-1}}=&\,x^3+(x+\delta)^{2^m+2}+(x+\delta)^{2^{m+1}+1}\\
		=&\,(x+\delta+\delta^{2^m})^3+\delta^3+\delta^{3\cdot 2^m},
	\end{align*}
	or \begin{equation}\label{1eqq/2+1}
		g(x)= x^{2^{m-1}} \circ  \left((x+\delta+\delta^{2^m})^3+\delta^3+\delta^{3\cdot 2^m}\right)\circ x^{(2^m+1)/3}.
	\end{equation}
	Moreover, since $3 \cdot (2^{m+1}-1)/3 \equiv 1 \pmod{2^m-1},$ we get 
	\begin{align*}
		&\,	\left((x+\delta+\delta^{2^m})^3+\delta^3+\delta^{3\cdot 2^m}\right)^{-1}\\=&\,\left((x+\delta^3+\delta^{3\cdot 2^m})\circ x^3 \circ (x+\delta+\delta^{2^m})\right)^{-1}\\
		=&\, (x+\delta+\delta^{2^m}) \circ x^{(2^{m+1}-1)/3} \circ (x+\delta^3+\delta^{3\cdot 2^m})\\
		=&\, (x+\delta^3+\delta^{3\cdot 2^m})^{(2^{m+1}-1)/3}+\delta+\delta^{2^m}.
	\end{align*}		
	Together with Eq. \eqref{1eqq/2+1} yields that  the compositional inverse of $g(x)$ over $\gf_{2^m}$ is 
	\begin{align*}
		g^{-1}(x)=&\,x^{3\cdot2^{m-1}} \circ  \left((x+\delta^3+\delta^{3\cdot 2^m})^{(2^{m+1}-1)/3}+\delta+\delta^{2^m}\right) \circ x^2\\
		=&\,\left((x^2+\delta^3+\delta^{3\cdot 2^m})^{(2^{m+1}-1)/3}+\delta+\delta^{2^m}\right)^{3\cdot2^{m-1}}.
	\end{align*}
	
	Consequently, taking $\phi(x)=x$ and $\bar{\phi}(x)=0$ in Lemma \ref{th1}, it implies by Lemma \ref{th1} that the compositional inverse of $f(x)$ over $\gf_{2^{2m}}$ is 
	$$f^{-1}(x)=x+\left((\Tr_m^{2m}(x)^2+\delta^3+\delta^{3\cdot 2^m})^{(2^{m+1}-1)/3}+\delta^{2^m}\right)^{2^{m-1}+1},$$
	which is the desired result. 
\end{proof}

In \cite{li2020some}, a proposition was presented regarding the  permutation property of the polynomial
$x+\left(\Tr_m^{2m}(x)^{(2^{m+1}-1)/3}+\delta\right)^{3}$ over $\gf_{2^{2m}}$, where $\delta \in \gf_{2^{2m}}$ and $m$ is odd.  We investigate the compositional inverse of this calss of permutation polynomial. 
\begin{theorem}
	Let	 $\delta \in \gf_{2^{2m}}$ and $m$ be odd. Then the compositional inverse of permutation polynomial 
	$$f(x)=x+\left(\Tr_m^{2m}(x)^{(2^{m+1}-1)/3}+\delta\right)^{3}$$
	over $\gf_{2^{2m}}$ is 		$$f^{-1}(x)=x+\left((\Tr_m^{2m}(x)+\delta^{2^{m+1}+1}+\delta^{2+2^m})^{(2^{m+1}-1)/3}+\delta^{2^m}\right)^{3}.$$
\end{theorem}

\begin{proof}
	Since $f(x)$ permutes $\gf_{2^{2^m}}$, we have that  $g(x)=x+\left(x^{(2^{m+1}-1)/3}+\delta\right)^{3}+\left(x^{(2^{m+1}-1)/3}+\delta\right)^{3\cdot2^m}$ is a permutation polynomial over $\gf_{2^m}$ by Lemma \ref{th1}. 
	Moreover,  we have 
	\begin{align*}
		g(x) \circ x^3=&\,x^3+(x+\delta)^{3}+(x+\delta)^{3\cdot 2^m}\\
		=&\,(x+\delta+\delta^{2^m})^3+\delta^{2^{m+1}+1}+\delta^{2+2^m}
	\end{align*}
	because of $3 \cdot (2^{m+1}-1)/3 \equiv 1 \pmod{2^m-1},$
	or, equivalently,  \begin{equation*}\label{1eq3}
		g(x)=\left((x+\delta+\delta^{2^m})^3+\delta^{2^{m+1}+1}+\delta^{2+2^m}\right) \circ  x^{(2^{m+1}-1)/3 }.
	\end{equation*}
	Consequently, the compositional inverse of $g(x)$ over $\gf_{2^m}$ is 
	\begin{align*}
		g^{-1}(x)=	&\, x^3 \circ	\left((x+\delta+\delta^{2^m})^3+\delta^{2^{m+1}+1}+\delta^{2+2^m}\right)^{-1}\\
		=&\,x^3 \circ \left((x+\delta^{2^{m+1}+1}+\delta^{2+2^m})\circ x^3 \circ (x+\delta+\delta^{2^m})\right)^{-1}\\
		=&\, x^3 \circ (x+\delta+\delta^{2^m}) \circ x^{(2^{m+1}-1)/3} \circ (x+\delta^{2^{m+1}+1}+\delta^{2+2^m})\\
		=&\, \left( (x+\delta^{2^{m+1}+1}+\delta^{2+2^m})^{(2^{m+1}-1)/3}+\delta+\delta^{2^m})\right)^3.
	\end{align*}			
	Hence, taking $\phi(x)=x$ and $\bar{\phi}(x)=0$ in Lemma \ref{th1},  it implies by Lemma \ref{th1} that the compositional inverse of $f(x)$ over $\gf_{2^{2m}}$ is 
	$$f^{-1}(x)=x+\left((\Tr_m^{2m}(x)+\delta^{2^{m+1}+1}+\delta^{2+2^m})^{(2^{m+1}-1)/3}+\delta^{2^m}\right)^{3}.$$
	We are done.
\end{proof}

In the work by \cite{li2020some}, a   proposition was established concerning the  permutation property of  the polynomial
$x+\left(\Tr_m^{2m}(x)^{2^{\frac{m+1}{2}-1}}+\delta\right)^{2^{\frac{m+1}{2}+1}}$ over $\gf_{2^{2m}}$, where $\delta \in \gf_{2^{2m}}$ and $m$ is odd. In the subsequent theorem, our focus revolves around exploring  the compositional inverse of this particular class of permutation polynomial.
\begin{theorem}
	Let	 $\delta \in \gf_{2^{2m}}$ and $m$ be odd. Then the compositional inverse of permutation polynomial 
	$$f(x)=x+\left(\Tr_m^{2m}(x)^{2^{\frac{m+1}{2}}-1}+\delta\right)^{2^{\frac{m+1}{2}}+1}$$
	over $\gf_{2^{2m}}$ is 			$$f^{-1}(x)=x+\left((\Tr_m^{2m}(x)+\delta^{2^{\frac{m+1}{2}}+2^m}+\delta^{2^{\frac{3m+1}{2}}+1})^{2^{\frac{m+1}{2}}-1}+\delta^{2^m}\right)^{2^{\frac{m+1}{2}}+1}.$$
\end{theorem}

\begin{proof}
	According to by Lemma \ref{th1}, 
	we have that  $$g(x)=x+\left(x^{2^{\frac{m+1}{2}}-1}+\delta\right)^{2^{\frac{m+1}{2}}+1}+\left(x^{2^{\frac{m+1}{2}}-1}+\delta\right)^{2^{\frac{m+1}{2}}\cdot 2^m+2^m}$$ is a permutation polynomial over $\gf_{2^m}$ if 	$f(x)$ permutes $\gf_{2^{2^m}}$. 
	Since $(2^{\frac{m+1}{2}}-1)(2^{\frac{m+1}{2}}+1)-2(2^m-1)=1,$ we obtatin 
	\begin{align*}
		g(x) \circ x^{2^{\frac{m+1}{2}}+1}=&\,x^{2^{\frac{m+1}{2}}+1}+(x+\delta)^{2^{\frac{m+1}{2}}+1}+(x+\delta^{2^m})^{2^{\frac{m+1}{2}}+1}\\
		=&\, x^{2^{\frac{m+1}{2}}+1}+(\delta+\delta^{2^m})x^{2^{\frac{m+1}{2}}}+(\delta+\delta^{2^m})^{2^{\frac{m+1}{2}}}x+\delta^{2^{\frac{m+1}{2}+1}}+\delta^{2^{\frac{3m+1}{2}}+2^m}\\
		=&\, (x+\delta+\delta^{2^m})^{2^{\frac{m+1}{2}}+1}+\delta^{2^{\frac{m+1}{2}}+2^m}+\delta^{2^{\frac{3m+1}{2}}+1},
	\end{align*}
	or, equivalently, 
	$$g(x)=\left((x+\delta+\delta^{2^m})^{2^{\frac{m+1}{2}}+1}+\delta^{2^{\frac{m+1}{2}}+2^m}+\delta^{2^{\frac{3m+1}{2}}+1}\right)\circ x^{2^{\frac{m+1}{2}}-1}.$$
	This implies that 
	the compositional inverse of $g(x)$ over $\gf_{2^m}$ is 
	\begin{align}\label{1eq2m+1/2-1}
		g^{-1}(x)=&\,x^{2^{\frac{m+1}{2}}+1}\circ \left((x+\delta^{2^{\frac{m+1}{2}}+2^m}+\delta^{2^{\frac{3m+1}{2}}+1})\circ x^{2^{\frac{m+1}{2}}+1}\circ (x+\delta+\delta^{2^m})\right)^{-1}\nonumber\\
		=&\, x^{2^{\frac{m+1}{2}}+1}\circ (x+\delta+\delta^{2^m})\circ x^{2^{\frac{m+1}{2}}-1} \circ (x+\delta^{2^{\frac{m+1}{2}}+2^m}+\delta^{2^{\frac{3m+1}{2}}+1})\nonumber\\
		=&\,\left((x+\delta^{2^{\frac{m+1}{2}}+2^m}+\delta^{2^{\frac{3m+1}{2}}+1})^{2^{\frac{m+1}{2}}-1}+\delta+\delta^{2^m}\right)^{2^{\frac{m+1}{2}}+1}.
	\end{align}		 
	Hence, taking $\phi(x)=x$ and $\bar{\phi}(x)=0$ in Lemma \ref{th1},  it implies by Lemma \ref{th1} and Eq. \eqref{1eq2m+1/2-1} that the compositional inverse of $f(x)$ over $\gf_{2^{2m}}$ is 
	$$f^{-1}(x)=x+\left((\Tr_m^{2m}(x)+\delta^{2^{\frac{m+1}{2}}+2^m}+\delta^{2^{\frac{3m+1}{2}}+1})^{2^{\frac{m+1}{2}}-1}+\delta^{2^m}\right)^{2^{\frac{m+1}{2}}+1}.$$
	We complete the proof.
\end{proof}

In \cite{li2020some}, a  proposition was presented regarding the permutation property of the polynomial
$x+\left(\Tr_m^{2m}(x)^k+\delta\right)^s$ over $\gf_{2^{2m}}$, where $\delta \in \gf_{2^m}$. It is observed that  the permutation polynomial of this form over $\gf_{2^{2m}}$ is an involution. Given the similarity in proof, we omit it here.

\begin{theorem}
	Let $\delta \in \gf_{2^m}$ and $k, s$ be positive integers.
	Then the polynomial $$f(x)=x+\left(\Tr_m^{2m}(x)^k+\delta\right)^s$$ is an involution over $\gf_{2^{2m}}.$
\end{theorem}

The work by \cite{li2020some} introduced a proposition concerning   the permutation property of the polynomial
$x+\left(\Tr_m^{2m}(x)^k+\delta\right)^{i(2^m+1)}$ over $\gf_{2^{2m}}$, where $\delta \in \gf_{2^{2m}}$ and $i$ is a positive integer with $i < 2^m-1$. It is noted  that the permutation polynomial of this form over $\gf_{2^{2m}}$ is an involution. Given the similarity the proof, we will omit it at this juncture.  

\begin{theorem}
	Let $\delta \in \gf_{2^{2m}}$ and $k, s$ be positive integers with  $i < 2^m-1.$ 	Then the polynomial $$f(x)=x+\left(\Tr_m^{2m}(x)^k+\delta\right)^{i(2^m+1)}$$ is an involution over $\gf_{2^{2m}}.$
\end{theorem}

In the research presented by \cite{li2020some}, a  proposition was outlined regarding  the permutation property of the polynomial
$x+\left(\Tr_m^{2m}(x)^k+\delta\right)^{2^i}+\left(\Tr_m^{2m}(x)^k+\delta\right)^{2^j}$ over $\gf_{2^{2m}}$, where $\delta \in \gf_{2^{2m}}$ and $i \neq j$. We will study the compositional inverse of this class of permutation polynomial over $\gf_{2^{2m}}$ in the following theorem.    Due to the similarity in the proof, we will omit it in this context. 

\begin{theorem}
	Let $\delta \in \gf_{2^{2m}}$ and $i\neq j.$ Then the compositional inverse of permutation polynomial $$f(x)=x+\left(\Tr_m^{2m}(x)^k+\delta\right)^{2^i}+\left(\Tr_m^{2m}(x)^k+\delta\right)^{2^j}$$  over $\gf_{2^{2m}}$ is 
	\begin{align*}f^{-1}(x)=&\,x+\left((\Tr_m^{2m}(x)+\delta^{2^i}+\delta^{2^{i+m}}+\delta^{2^j}+\delta^{2^{j+m}})^k+\delta\right)^{2^i}\\&\,+\left((\Tr_m^{2m}(x)+\delta^{2^i}+\delta^{2^{i+m}}+\delta^{2^j}+\delta^{2^{j+m}})^k+\delta\right)^{2^j}.\end{align*}
\end{theorem}

In \cite{li2020some}, a  proposition were presented regarding  the permutation property of the polynomials
$x+\left(\Tr_m^{2m}(x)^k+\delta\right)^{2^i+1}+\left(\Tr_m^{2m}(x)^k+\delta\right)^{2^m+2^i}$ and $x+\left(\Tr_m^{2m}(x)^k+\delta\right)^{2^i+1}+\left(\Tr_m^{2m}(x)^k+\delta\right)^{2^{m+i}+1}$ over $\gf_{2^{2m}}$, where $\delta \in \gf_{2^{2m}}$. We aim to provide their compositional inverses over $\gf_{2^{2m}}.$

\begin{theorem}
	Let $\delta \in \gf_{2^{2m}}$. For a integer $i$ with $0< i< m,$
	the compositional inverses  of  $f_2(x)=x+\left(\Tr_m^{2m}(x)^k+\delta\right)^{2^i+1}+\left(\Tr_m^{2m}(x)^k+\delta\right)^{2^m+2^i}$  and $f_3(x)=x+\left(\Tr_m^{2m}(x)^k+\delta\right)^{2^i+1}+\left(\Tr_m^{2m}(x)^k+\delta\right)^{2^{m+i}+1}$ over $\gf_{2^{2m}}$ are 
	\begin{align*}
		f_2^{-1}(x)=x+\left((\Tr_m^{2m}(x)+(\delta+\delta^{2^m})^{2^i+1})^k+\delta\right)^{2^i+1}\\+\left((\Tr_m^{2m}(x)+(\delta+\delta^{2^m})^{2^i+1})^k+\delta\right)^{2^i+2^m}
	\end{align*}    and \begin{align*}
		f_3^{-1}(x)=x+\left((\Tr_m^{2m}(x)+(\delta+\delta^{2^m})^{2^i+1})^k+\delta\right)^{2^i+1}\\+\left((\Tr_m^{2m}(x)+(\delta+\delta^{2^m})^{2^i+1})^k+\delta\right)^{2^{m+i}+1},
	\end{align*}  respectively.  
\end{theorem}
\begin{proof}
	As the proofs are similar, we will focus solely on  determining the compositional inverse of $f_2(x)$ over  $\gf_{2^{2m}}$ here. 
	It follows from Lemma \ref{th1} that $f_2(x)$ permutes $\gf_{2^{2m}}$ if and only if $g(x)=x+(x^k+\delta)^{2^i+1}+(x^k+\delta)^{2^m(2^i+1)}+(x^k+\delta)^{2^m+2^i}+(x^k+\delta)^{1+2^{m+i}}$ permutes $\gf_{2^m}.$
	Moreover, we have 
	\begin{align*}
		g(x)=&\,x+(x^k+\delta)^{2^i+1}+(x^k+\delta)^{2^m(2^i+1)}+(x^k+\delta)^{2^m+2^i}+(x^k+\delta)^{1+2^{m+i}}\\
		=&\, x+(x^k+\delta)^{2^i+1}+(x^k+\delta^{2^m})^{2^i+1}+(x^k+\delta^{2^m})(x^k+\delta)^{2^i}\\
		&+(x^k+\delta)(x^k+\delta^{2^m})^{2^i}\\
		=&\, x+(x^k+\delta)^{2^i}(\delta+\delta^{2^m})+	(x^k+\delta^{2^m})^{2^i}(\delta+\delta^{2^m})\\
		=&\, x+(\delta+\delta^{2^m})^{2^i+1}.	\end{align*}
	Thus, $g^{-1}(x)=x+(\delta+\delta^{2^m})^{2^i+1}.$
	This yields that the compositional inverse of $f_2(x)$ over $\gf_{2^{2m}}$ is 
	\begin{align*}
		f_2^{-1}(x)=x+\left((\Tr_m^{2m}(x)+(\delta+\delta^{2^m})^{2^i+1})^k+\delta\right)^{2^i+1}\\+\left((\Tr_m^{2m}(x)+(\delta+\delta^{2^m})^{2^i+1})^k+\delta\right)^{2^i+2^m}
	\end{align*} by Lemma \ref{th1}.
	We are done.
\end{proof}

\section*{Declarations}

\begin{itemize}
	\item
	The research of Pingzhi Yuan is partially supported by the National Natural Science Foundation of China (Grant No. 12171163). The research of Danyao Wu is partially supported by the Guangdong Basic and Applied Basic Research Foundation (Grant No. 2020A1515111090).
\end{itemize}

\bibliography{sn-bibliography}

\end{document}